\documentclass[oneside,reqno]{amsart}

\usepackage[T2A]{fontenc}
\usepackage{amsmath,amssymb,amsthm}
\usepackage[english]{babel}
\usepackage{mmap}
\usepackage[pdftex]{graphicx}
\usepackage{float}
\usepackage{xcolor}
\usepackage[unicode, pdftex]{hyperref}
\theoremstyle{definition}

\newtheorem{Definition}{Definition}

\newtheorem{theorem}{Theorem}[section]

\newtheorem{lemma}{Lemma}[section]

\theoremstyle{definition}
\newcommand{\Z}{\mathbb{Z}}
\newcommand{\R}{\mathbb{R}}

\newcommand{\Eis}{\mathit{Eis}}
\newcommand{\Prim}{\mathrm{Prim}}
\newcommand{\Vol}{\mathrm{Vol}}
\renewcommand{\L}{\textrm{\foreignlanguage{russian}{Л}}}

\makeatletter
\newcommand{\extp}{\@ifnextchar^\@extp{\@extp^{\,}}}
\def\@extp^#1{\mathop{\bigwedge\nolimits^{\!#1}}}
\makeatother

\title[Convex triangulations of $\R P^2$]{Equilateral convex triangulations of $\mathbb R P^2$ with three conical points of equal defect}

\date{\today}
\begin{document}
\begin{abstract}
Consider triangulations of $\mathbb R P^2$ whose all vertices have valency six except three vertices of valency $4$. In this chapter we prove that the number $f(n)$ of such triangulations with no more than $n$ triangles grows as $C\cdot n^2+ O(n^{3/2})$ where $C = \frac{1}{20} \sqrt{3} \cdot\L( \frac{\pi}{3} ) \zeta^{-1}(4) \zeta(\Eis, 2) \approx 0.2087432125056015...$, where $\L$ is the Lobachevsky function and $\zeta(\Eis,2) =\sum\limits_{(a,b)\in\mathbb Z^2\setminus 0}{\frac{1}{|a+b\omega^2|^4}}$, and $\omega^6=1$.
\end{abstract}

\maketitle
\centerline{Mikhail Chernavskikh\footnote{Lomonosov Moscow State University, 1 Leninskije gory, Moscow 119991, Russia, e-mail: mike.chernavskikh@gmail.com
ORCID: 0000-0001-9513-7891}, Altan Erdnigor\footnote{HSE University, Russian Federation, Department of Mathematics, 6 Usacheva st, Moscow 119048, Russia, E-mail address: alerdnigor@gmail.com
ORCID: 0000-0002-4617-5231}, Nikita Kalinin\footnote{Saint Petersburg State University, 7/9, Unversitetskaya emb., 199034, Saint Petersburg, Russia, email: nikaanspb@gmail.com, ORCID: 0000-0002-1613-5175}, Alexandr Zakharov\footnote{Saint Petersburg State University, 7/9, Unversitetskaya emb., 199034, Saint Petersburg, Russia, email: Zaa060998@gmail.com, ORCID: 0000-0002-0305-099X}}

AMS codes: 51M09, 57N45, 11P21, 11M36, 11E45.

\keywords{flat metric, equilateral triangulation, conical singularity, zeta function, Epstein zeta function, hyperbolic volume}

\section{Introduction}

Consider a triangulation\index{triangulation} $T$ of $\mathbb R P^2$ such that each vertex of $T$ is contained in at most six triangles. These triangulations are called {\it convex}\index{convex triangulation}. Let each triangle in $T$ be the equilateral triangle with sides of length one. This supplies $\mathbb R P^2$ with a flat metric outside of the vertices of $T$. If at a vertex $v$ of $T$ exactly $k$ triangles come together then we say that the defect at $v$ is equal to $(6-k)\pi/3$. Convex triangulations are exactly those with non-negative defects. By counting edges, vertices, and triangles in $T$ one can see that the sum of all defects of the vertices of $T$ is equal to $2\pi$, because the Euler characteristic of $\mathbb R P^2$ is one. Hence this construction gives a flat metric on $\mathbb R P^2$ except at most six points (vertices of valency less than six).

Consider the covering of $\mathbb R P^2$ by $S^2$. Naturally we obtain a metric on $S^2$ which is flat everywhere except at most twelve points with so-called conical singularities\index{conical singularity}. A conical singularity with defect $\theta, 0\leq \theta< 2\pi$ is locally modeled on the sector $0\leq \phi \leq 2\pi-\theta$ of the unit disk $(r,\phi)$ with identified boundaries $(r,0)\sim(r,2\pi-\theta)$. 

By Alexandrov's theorem \cite{1942} each flat metric on $S^2$ with conical singularities can be realized as the surface of a (possible degenerate) convex polytope in $\mathbb R^3$ with intrinsic metric. If we have only two conical points (with defects $\theta$ and $2\pi-\theta$) on an everywhere else flat $\mathbb R P^2$, then its covering $S^2$ has four conical points, and they should be identified by the central symmetry. Thus this metric is realized as a two-sided planar parallelogram (a degenerate polytope) with angles $\theta/2, \pi-\theta/2$. If we consider $\mathbb R P^2$ with three conical points, then its covering $S^2$ is isometric to a centrally-symmetric octahedron.

Thurston \cite{thurston1998shapes} (see also the lecture notes \cite{schwartz2015notes} which contain more detailed proofs) studied convex triangulations of $S^2$ and the moduli space of flat metrics on $S^2$ with a finite number of arbitrary conical singularities; the set of convex equilateral triangulations lives as a discrete subset in this moduli space. Following Thurston's ideas, we study equilateral triangulations of $\mathbb R P^2$ whose vertices have all valency six except three vertices of valency four. 

\section{Triangulations of $\R P^2$ with three marked points with defects $2\pi/3$}
A graph without loops and multiple edges, drawn on $\R P^2$, is called {\it a triangulation} of $\R P^2$ if each face of this graph has three edges. Note that two faces of such a triangulation can intersect in zero, one, two, or three vertices.

Consider a triangulation $T$ of $\R P^2$ such that only three vertices $A,B,C$ have valency four, and all the other vertices have valency six.  

$T$ gives a flat metric $\mu_{\R P^2}$ on $\R P^2$ except at $A,B,C$. Passing to the universal covering sphere $S^2$ one gets a flat metric $\mu_{S^2}$ on $S^2$ except six points. By Alexandrov's theorem, $\mu_{S^2}$ is realised as the intrinsic metric of the surface of a certain centrally symmetric octahedron $F$. The projections of the edges of $F$ give six geodesic paths between $A,B,C$ in $\R P^2$, thus cutting $\R P^2$ into four triangles (all with vertices $A,B,C$, so we have four triangles $ABC$). Choose one of these four triangles, call it $\Delta$. Call $A,B,C, \Delta$ the {\it label} of $T$.

Denote by $T_{\R P^2}$ the set $\{T,A,B,C,\Delta\}$ of labelled triangulations of $\R P^2$. Two such triangulations $(T_1,A_1,B_1,C_1,\Delta_1),(T_2,A_2,B_2,C_2,\Delta_2)$ are said {\it isometric} if there exists a map between triangulations $T_1,T_2$, which sends vertices and edges of $T_1$ to vertices and edges of $T_2$, $A_1$ to $A_2$, $B_1$ to $B_2$, $C_1$ to $C_2$ and $\Delta_1$ to $\Delta_2$.

Consider the smallest possible triangulation of $\R P^2$ which consists of three vertices, four triangles, and six edges. We can label it in $3\cdot 2\cdot 1 \cdot 4$ different ways, but all the obtained labelled triangulations are isometric.

Let $f(n)$ be the cardinality of the set of isometry classes of labelled triangulations in $\R P^2$ with no more than $n$ triangles.

Consider a labelled triangulation $\{T,A,B,C,\Delta\}$ of $\R P^2$. Consider the octahedron $F$ as above. Then $\Delta$ lifts to $F$ as two triangles $\Delta_1,\Delta_2$. Call $A,B,C$ the vertices of $\Delta_1$ and $A',B',C'$ the vertices of $\Delta_2$, then $A,A'\in F$ are projected to $A\in \R P^2$, $B,B'\in F$ are projected to $B\in \R P^2$, $C,C'\in F$ are projected to $C\in \R P^2$ under the covering map $F\to \R P^2$. Who is $\Delta_1$ and who is $\Delta_2$ is uniquely defined by the condition that the order of vertices $A,B,C$ is counterclockwise (looking from outside of $F\subset \R^3$, see Figure~\ref{fig_1}). 

We can reverse the procedure. Consider a convex triangulations $\tilde T$ of $S^2$ with six points with defects $2\pi/3$. Mark three of these points as $A,B,C$ and suppose that by supplying $S^2$ with a flat metric as above and realising it as the surface of a polyhedron we obtain a centrally symmetric octahedron $F$, and $ABC$ is a face of $F$, and its orientation gives the counterclockwise order of $ABC$ (Figure~\ref{fig_1}). The central symmetry of $F$ preserves $\tilde T$ and provides us with a projection $p:F\to \R P^2$. Projecting $\tilde T$ to a triangulation $T$ of $\R P^2$ we mark the images of $A,B,C\in F$ as $A,B,C\in \R P^2$. Label by $\Delta$ the image of the face $ABC$ of $F$ under $p$. 

Consider a centrally symmetric octahedron\index{octahedron} $F\subset \R^3$, such that the sum of angles at each vertex of $F$ is $4\pi/3$. Suppose that $\tilde T$ is a convex  equilateral triangulation of $F$. Choose any face of $F$ and call its  vertices $A,B,C$ in such a way that the order of $A,B,C$ is counterclockwise (if looking from outside of $F\subset \R^3$, see Figure~\ref{fig_1}) and call the opposite faces $A',B',C'$. $(A,B,C)$ is a label of $\tilde T$.  We consider labeled triangulations $(\tilde T,A,B,C)$ up to isometry.

We proved the following lemma

\begin{lemma} There exist a bijection between labelled triangulations $(T,A,B,C,\Delta)$ of $\R P^2$ with $n$ triangles and labelled triangulations $(\tilde T,A,B,C)$ with $2n$ triangles.
\end{lemma}

Therefore $f(n) = \# \{(\tilde T,A,B,C)$  with no more than $2n$ triangles$\}$.

\section{Moduli space of flat metrics on $S^2$ with six pair-wise centrally symmetric  conical points of  equal defect}

Consider the set of all centrally symmetric octahedra $F$, such that the sum of angles at each vertex of $F$ is $4\pi/3$. There exist natural coordinates on this space as follows \cite{wang2021shapes}. 

Recall that for a triangle $ABC$ whose angles are all less than $2\pi/3$ the Fermat--Torricelli point is the unique point $X$ inside the triangle such that all the angles $AXB,BXC,CXA$ are equal to $2\pi/3$. If the angle $ABC$ is equal to $2\pi/3$ then we say that $B$ is the Fermat--Torricelli point of the triangle $ABC$. The Fermat--Torricelli point $X$ is the point minimizing $|XA|+|XB|+|XC|$.

Pick the Fermat--Torricelli point in each face of $F$ and connect it with the vertices of this face. Then, among the lengths of these 24 intervals there are only four different ones  \cite{wang2021shapes}, let us denote them by $a,b,c,d$.  

\begin{figure}[h]
\includegraphics[scale=0.3]{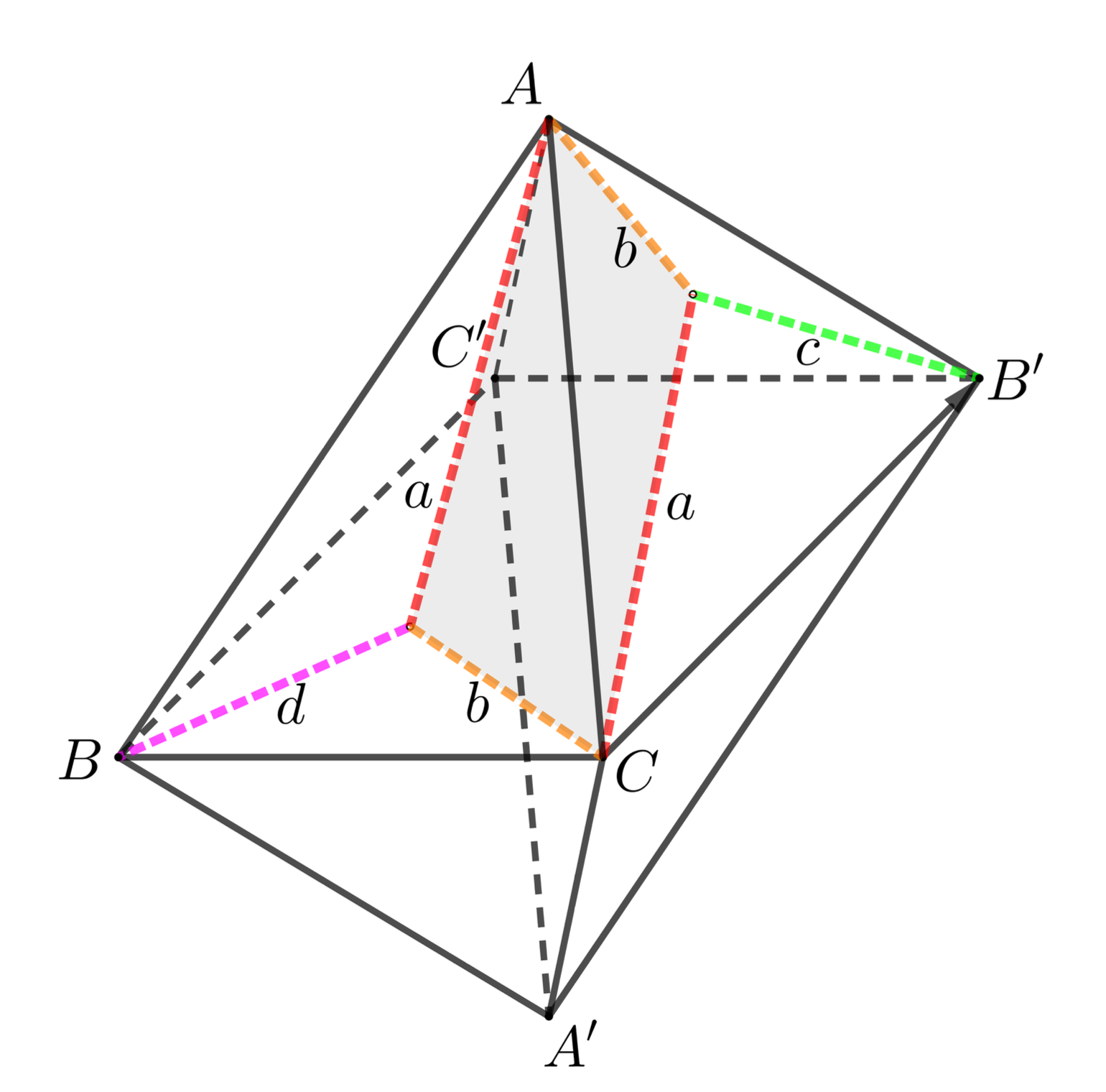}
\caption{An octahedron with vertices $AA'BB'CC'$, two Fermat--Torricelli points in the faces $ABC, ACB'$, and the corresponding parallelogram with sides $a,b$. }
\label{fig_1}
\end{figure}

Conversely, given four non-negative numbers $a,b,c,d$ (we allow at most one of them to be zero, see the examples below), we can construct 12 parallelograms with acute angle $\pi/3$ and sides $(a,b),(a,c),(a,d),(b,c),(b,d),(c,d)$ (two copies of each parallelogram). Let us bend each of them along its diagonal and glue them in an octahedron $F$. The diagonals of the parallelograms become edges of $F$. If $a=0$ then the angle $BAC$ is $2\pi/3$ and the parallelogram in Figure~\ref{fig_1} degenerates to the edge $AC$ of $F$.

Let us say that the counterclockwise order around $A$ of the intervals from $A$ to the Torricelli points of the adjacent faces gives $a,b,c,d$, and let us also fix that the interval of length $a$ belongs to the triangle $ABC$, see Figure~\ref{fig_1}. Note also that there are six rotational orderings of $a,b,c,d$ and that all of them are realized at exactly six vertices of $F$.

Given these coordinates on the moduli space of such octahedra (see details in \cite{wang2021shapes} for octahedra with general defects at the vertices), we easily compute the area of the octahedron $F=(a,b,c,d)$ (note that $2\sin{\frac{\pi}{3}} =\sqrt{3}$): it is $$Area(a,b,c,d)=\sqrt{3}(ab+ac+ad+bc+bd+cd).$$

Let $Q(a,b,c,d)=ab + ac + ad + bc + bd + cd; Q$ is a quadratic form of signature $(1,3)$ since
$$Q(a,b,c,d)=\frac{1}{8}(3(a+b+c+d)^2-(c+d-a-b)^2-2(c-d)^2-2(a-b)^2).$$

If we start with an equilateral triangulation of $\R P^2$ with $n$ triangles, then its covering sphere has $2n$ triangles, and each triangle has area $\sqrt 3/4$, so the area of the sphere is $$n\sqrt 3/ 2 = \sqrt{3}Q(a,b,c,d))$$ which gives $n=2Q(a,b,c,d)$. 

Define 
$$X=\left\{ 
		(a, b, c, d)\in \R^4_{\geq0} \mid
		 Q(a,b,c,d)\le 1 
	\right\}.$$
	
\begin{lemma}
\label{lem_vol}
\[
	\Vol (X) = \sqrt 3 \L ( \frac{\pi}{3}),
\]
where $\L(\phi) = - \int_0^\phi \ln |2 \sin \theta| d \theta$ is the Lobachevsky function\index{Lobachevsky function}.
\end{lemma}
\begin{proof}

Recall that each bilinear symmetric form $( \cdot , \cdot )$ yields a volume form on $\mathbb R^4$.
Namely, $\mathrm{Vol}_{( \cdot , \cdot )  }(v_1, v_2, v_3, v_4) = \pm \sqrt{|\det (v_i, v_j)|}$, the square root of the Gramian of $( \cdot  , \cdot )$ with respect to this system of vectors. 
The sign of the (oriented) volume is defined by the orientation of $(v_1, v_2, v_3, v_4)$.

Denote by $\bar Q$ the bilinear symmetric form associated with $Q$.
Define a $3$-form $\alpha$ on $\R^4$ as follows:
\[\forall x\in \R^4, \alpha: \extp^3 T_x\R^4\to \R,  \alpha(v_1, v_2, v_3) = 
\mathrm{Vol}_{\bar Q}(x, v_1, v_2, v_3) 
.\]

Note that $Q$ induces a hyperbolic structure in the set $Q(v)=1$ and that $\alpha$ is the corresponding volume form. 
Next (see \cite{wang2021shapes} for details),
$$\int\limits_{v\in \R_{\geq 0}^4, Q(v) = 1}\alpha=3 \L(\frac{\pi}{3}).$$


Let $dQ$ be the differential of $Q$, namely

$$dQ:T_x\R^4\to \R, dQ(w) = 2 \bar Q (x, w).$$

Let $v=(a,b,c,d)$, consider $Q'(v)=Q'(a, b, c, d) = a^2 + b^2 + c^2 + d^2$.
Let $\omega$ be the standard Euclidian volume form $\omega(v_1, v_2, v_3, v_4) = \mathrm{Vol}_{\bar Q'}(v_1, v_2, v_3, v_4)$.

Let us prove that
\begin{equation}
\label{eq_dq}
	Q^{-1} dQ \wedge \alpha = 
	\frac{\sqrt 3}{2} \omega.
\end{equation}
Denote the coordinate basis in $\mathbb R^4$ by $(e_1, e_2, e_3, e_4)$. 
Take any $x, v_1, v_2, v_3 \in \mathbb R^4$, and denote by $A \in \mathrm{Mat}_{4 \times 4}(\mathbb R)$ the matrix of their coordinates.

On $T_x\R^4$ we have
\[
	(Q^{-1} dQ \wedge \alpha)(x, v_1, v_2, v_3) = 
	Q(v)^{-1} 2 \bar Q(x, x) \mathrm{Vol}_Q(x, v_1, v_2, v_3) = \]
	
	\[=2 \det A \mathrm{Vol}_Q(e_1, e_2, e_3, e_4) = 
	\frac{\sqrt 3}{2} \omega(x, v_1, v_2, v_3) \]

which proves \eqref{eq_dq}.

Now,
$$
\mathrm{Vol} (X) = \int\limits_X \omega =  
\frac{2}{\sqrt 3} \int\limits_X Q^{-1} dQ \wedge \alpha =$$

$$=\frac{2}{\sqrt 3} \int\limits_{0}^{1} q^{-1}dq
\int\limits_{a, b, c, d \ge 0, Q(a, b, c, d) = q} \alpha = 
\frac{2}{\sqrt 3} \int\limits_{0}^{1} q^{-1}(dq)~ q^2 
\int\limits_{a, b, c, d \ge 0, Q(a, b, c, d) = 1}\alpha = $$
$$=\frac{2}{\sqrt 3} \int\limits_{0}^{1} q dq \cdot
3 \L(\frac{\pi}{3}) = 
\sqrt 3
\L(\frac{\pi}{3}) 
.$$
 \end{proof}

 Denote 
\[
	g(n) = \# \left\{ (a, b, c, d)\in \Z^4_{>0} \mid Q(a,b,c,d) \le n \right\} 
.\] 

\begin{theorem}
\label{th_3}
\[
	g(n) = \sqrt{3} \L( \frac{\pi}{3} ) n^2 + O(n^{3/2})
,\] 	
where $$\sqrt{3} \L( \frac{\pi}{3} )
\approx 0.58597680967236472265039057221806926727385075240896...$$

\end{theorem}

\begin{proof}
Define 
	$$Y_t=\left\{ 
		(a, b, c, d)\in \R^4_{\geq0} \mid
		1\le Q(a,b,c,d) \le t 
	\right\}.$$
	Note that $g(n) = |Y_n\cap \mathbb Z_{>0}^4|$. It follows from Lemma~\ref{lem_vol} that 
$$
	g(n) \approx \Vol(Y_n) \approx  \sqrt{3} \L( \frac{\pi}{3} )n^2. 
$$

Note that the error term is proportional to the Euclidean three-dimensional volume of the boundary of $Y_n$ since the three-dimensional volume of the boundary of $X$ is finite (one can use a similar reasoning as in Lemma~\ref{lem_vol}). 

For $t\geq 1,$ denote the three-dimensional volume of the boundary of $Y_t$ by $r(t)$. Denote by $2Y_t$ the image of $Y_t$ under the homothety with center at $0$ and coefficient $2$. Then the three-dimensional volume of the boundary of $2Y_t$ is $8r(t)$. On the other hand $Y_{4t} = Y_4\cup 2Y_t$ hence $r(4t)\leq r(4)+ 8r(t)$, thus $$r(4t)+\frac{1}{7}r(4)\leq 8\big[r(t)+\frac{1}{7}r(4)\big].$$ Letting $b(t)=r(t)+\frac{1}{7}r(4)$ we obtain $b(4t)\leq 8b(t)$ and this  leads to the estimate $b(4^kt)\leq 8^kb(t)$. Let $n=4^kx, 1\leq x<4$. Note that $8^k\leq n^{3/2}$. Then $b(n)\leq 8^kb(x)$. Let $c=\max_{1\leq x\leq 4} b(x)$. Thus we obtain $b(n)\leq cn^{3/2}$. This can be rewritten as $r(n)+\frac{1}{7}r(4)\leq  c n^{3/2}$ and so the volume of the boundary of $Y_n$ is $O(n^{3/2})$.

\end{proof}

Let us also introduce 
\[
h(n) = \# \left\{ (a, b, c, d) \in \Z^4_{>0} \mid a \equiv b \equiv c \equiv d \pmod 3,  Q(a,b,c,d) \le n \right\} 
.\] 
The covolume (in $\Z^4$) of the lattice generated by such quadruples is $27$, so, repeating the arguments of our proof of Theorem~\ref{th_3} we obtain
\begin{theorem}
\[
	h(n) = \frac{\sqrt{3}}{27}\L( \frac{\pi}{3} ) n^2 + O(n^{3/2})
.\] 
\end{theorem}

\section{A parametrization of equilateral triangulations of $S^2$ with six centrally-symmetric points with defects $2\pi/3$}

%


Let $\omega = e^{ \frac{2 \pi i}{6}}=\frac{1+\sqrt{-3}}{2}$. Consider the Eisenstein lattice\index{Eisenstein lattice} $$\Eis = \mathbb Z\oplus \mathbb Z\omega^2\subset \mathbb C.$$

Define $\widetilde {Eis} =\frac{1}{1 - \omega^2}\Eis $. Note that $\widetilde {Eis}$ contains $\Eis$, and $\widetilde {Eis}\setminus \Eis$ is the set $z+\Eis$ where $z=\frac{1+\omega}{3}=\frac{1}{1 - \omega^2}$ is the Torricelli point of the triangle with  vertices $0,1, \omega$.

Consider a labelled  triangulation $(\tilde T,A,B,C)$ of a centrally symmetric octahedron with vertices  $A,A',B,B',C,C'$. Take the faces $ABC,ACB',AB'C',AC'B$, make a cut along $AC'$, and develop the obtained polygon onto the plane such that $A$ goes to $0\in\mathbb \Eis$ under our developing map, and the vertices of $T$ go to $\Eis$. Then the developing map is defined up to the action of $\mathbb Z_6$ by rotations, because under the developing map we preserve the local orientation at $A$.

\begin{figure}
\includegraphics[scale=0.2]{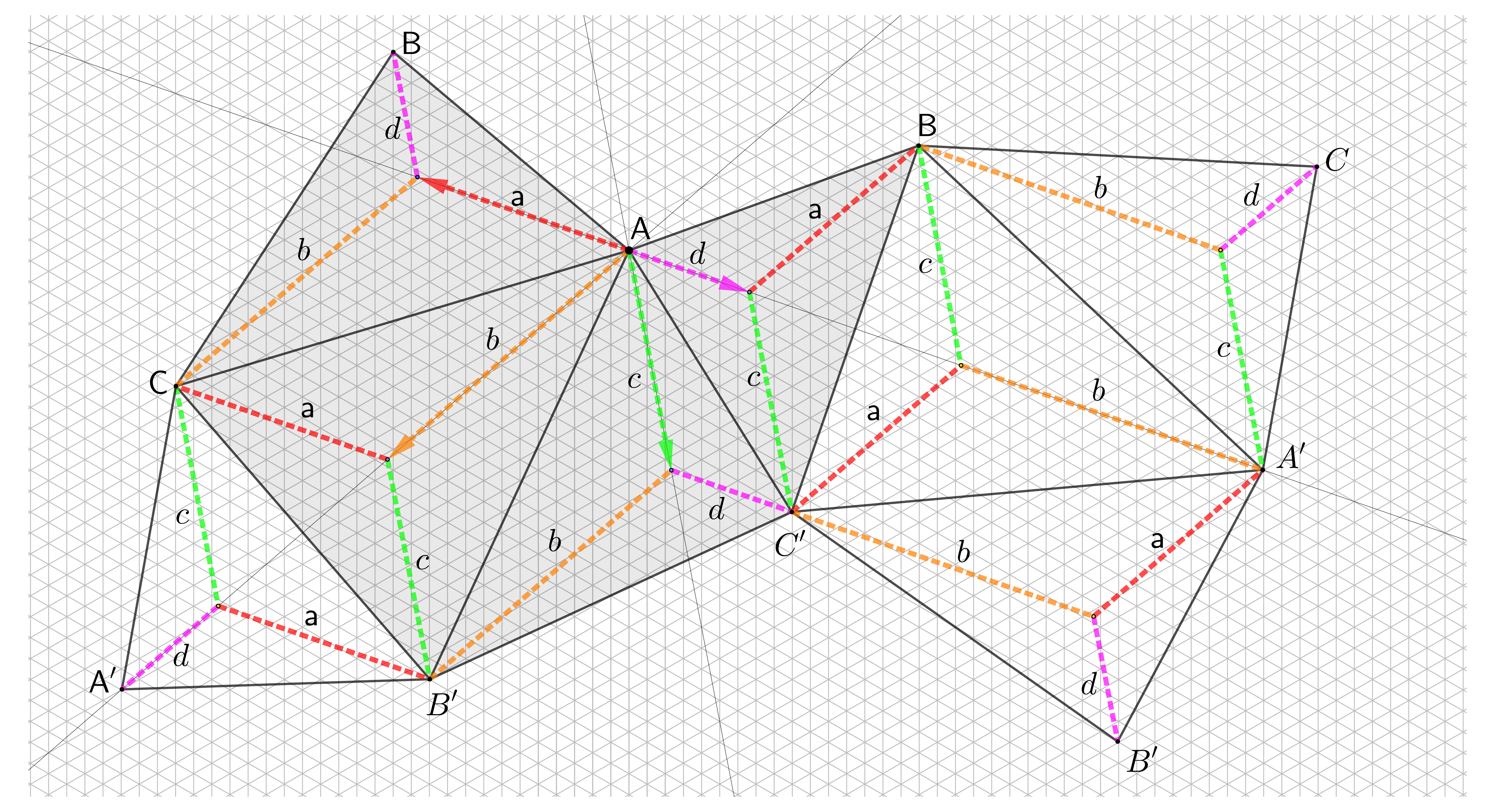}
\caption{A developing of the octahedron $AA'BB'CC'$ on $\R^2$ is presented, $A=0$, vertices $A,B,C,A',B',C'$ go to the lattice $\Eis$. Note that a triangulation of $\R P^2$ can be obtained from the grey area by gluing $AB$ to $AB$, then $BC$ to $B'C'$ and then $BC'$ to $B'C$. Note that the Torricelli centers of the faces do not belong to $\Eis$ but belong to $\widetilde {Eis}$, e.g. see the Torricelli center of $ABC$.}
\label{fig_dev}
\end{figure}

Let $\vec a, \vec b, \vec c, \vec d$ be the vectors in $\mathbb C$ connecting the point $0\in\Eis$ and the Torrichelli points of the four faces $ABC, ACB', AB'C',AC'B$ of $F$. 

\begin{lemma} 
Under the developing map vectors $\vec a, \vec b, \vec c, \vec d$ go to $\widetilde {Eis}$.
\end{lemma}
\begin{proof}
All the vertices of the octahedron are developed into the lattice points. Then the sums $\vec a + \vec b, \vec b + \vec c, \vec c + \vec d, \vec a + \vec d\omega^2$ also belong to $\Eis$. Then, $\vec a + \vec b - (\vec b + \vec c) + \vec c + \vec d - (\vec a + \vec d\omega^2) = \vec d(1 - \omega^2) \in \Eis$, hence $\vec d\in \frac{1}{1 - \omega^2}\Eis = \widetilde {Eis}$. Then, $\vec c+\vec d\in\Eis$ and the latter is a sublattice in $\widetilde {Eis}$, therefore $\vec c\in\widetilde {Eis}$. Similarly, $\vec b,\vec a\in \widetilde {Eis}$.
\end{proof}

\begin{Definition}
The vectors $$\vec e_x = {1}/(1 - \omega^2), \vec e_y = {\omega^2}/(1 - \omega^2)$$ form a basis in the lattice $\widetilde{Eis}$. Each vector in $\widetilde{Eis}$ can be expressed as $x\vec e_x + y\vec e_y$, $(x, y) \in \mathbb Z^2$. There are three cases for the sum $(x + y) \mod 3$. The lattice $\widetilde{Eis}$ is divided into three subsets:  $$\widetilde{Eis}_k = \{x\vec e_x+y\vec e_y\ in\  \widetilde{Eis}| (x + y) \equiv k \mod 3\}.$$ 

\end{Definition}
Note that $\vec e_x -\vec e_y = {1}/(1 - \omega^2)-{\omega^2}/(1 - \omega^2) = 1\in\Eis$ and $$2\vec e_x +\vec e_y={2}/(1 - \omega^2)+{\omega^2}/(1 - \omega^2) = \frac{2+\omega^2}{1 - \omega^2} = \frac{2+\omega^2}{1 - \omega^2}=\omega\in\Eis.$$  This implies that $\widetilde{Eis}_0 = Eis$.

However the vectors $\vec a, \vec b, \vec c, \vec d$ are not arbitrary. 
\begin{lemma}
If the vectors $\vec a, \vec c$ lie in $\widetilde{Eis}_k$ then $\vec b, \vec d$ lie in $\widetilde{Eis}_{-k}$.
In other words, there are three cases:
\begin{enumerate}
    \item $\vec a, \vec b ,\vec c ,\vec d \in \widetilde{Eis}_0$; 
    \item $\vec a, \vec c  \in \widetilde{Eis}_1$ and $\vec b, \vec d  \in \widetilde{Eis}_2$;
    \item $\vec a, \vec c  \in \widetilde{Eis}_2$ and $\vec b, \vec d  \in \widetilde{Eis}_1$.
\end{enumerate}
\end{lemma}

\begin{proof}
This follows from $\widetilde{Eis}_k + \widetilde{Eis}_m = \widetilde{Eis}_{k + m}$ and fact that $\omega^2\widetilde{Eis}_k = \widetilde{Eis}_k$.
\end{proof}

Thus we constructed a bijection between the labelled triangulations $(\tilde T, A,B,C)$ up to isometry and certain 4-tuples of vectors $\vec a, \vec b, \vec c, \vec d \in \widetilde{Eis}$ up to a $\Z_6$ action.

One could consider the sublattice $\Eis_0 \subset \Eis$, $$\Eis_0 = \left\{ x + y \omega^2 \mid x, y \in \Z , x + y \equiv 0 \pmod 3 \right\} .$$
The cosets of $\Eis_0$ in $\Eis$ are $\Eis_0, \Eis_1, \Eis_2$ where 
$$\Eis_k = \left\{ x + y \omega^2 \mid x, y \in \Z , x + y \equiv k \pmod 3 \right\} .$$

If $L$ is a lattice, let $\Prim L = \left\{ v \in L \setminus 0 \mid \nexists w \in L, n > 1 : n w = v \right\} $.

Let $E_0=\Prim \Eis \cap \Eis_0$ and $E_{\ne 0} =\Prim \Eis \cap (\Eis_1 \sqcup \Eis_2)$, then
\begin{equation}
	\label{very_important_fact}
	\Prim \Eis_0 = E_0 \sqcup 3 E_{\ne 0}.
\end{equation}
Indeed, it follows from $\Prim \Eis = E_0 \sqcup E_{\ne 0}$ that each primitive vector $v$ of $\Eis_0$ is either a primitive vector in $\Eis$ (and then it is an element of $E_0$) or there exists $v'\in \Eis, v=kv', k>1$ and $v'\notin \Eis_0, v'\in E_{\ne 0}$. In the latter case, $3v'\in \Eis_0$ (this is true for each vector in $\Eis$), therefore $k$ can be equal to three only. Therefore $v\in 3 E_{\ne 0}$.

\begin{theorem}

	$$f(n) = \frac{1}{6} 
	\#\{(z\in \Prim \Eis,(a, b, c, d)\in \Z^4_{>0})| \frac{2}{3} |z|^2 Q(a,b,c,d) \le n \},$$
	
where i)$z \in E_0 $ and $a,b,c,d$ are arbitrary  or ii)$z \in E_{\ne 0} , a \equiv b \equiv c \equiv d \pmod 3$.

\end{theorem}

\begin{proof}
	Each labelled  triangulation $(\tilde T,A,B,C)$  is determined by the vectors  $$\vec a, \vec b, \vec c, \vec d \in \frac{1}{1 - \omega^2} \Eis=\widetilde \Eis$$ with the oriented angles $ \angle (\vec a, \vec b) = \angle (\vec b, \vec c) =  \angle (\vec c, \vec d) =  \frac{\pi}{3}$.
One could find $z' \in \Prim \widetilde \Eis$ --- the primitive vector proportional to $\vec a$.
Then $$(\vec a, \vec b, \vec c, \vec d) = z'\cdot(a, b \omega , c \omega^2 , d \omega^3), a, b, c, d \in \Z_{>0}.$$

Let $z = (1 - \omega^2) z' \in \Prim \Eis$.
The number of triangles in $\tilde T$ is equal to the total area of the octahedron divided by the area of one equilateral triangle.
The area equals $\sin {\frac{\pi}{3}} \cdot 2 |z'|^2 Q(a,b,c,d) = \frac{1}{\sqrt 3} |z|^2 Q(a,b,c,d)$ whereas the area of one equilateral triangle is $\frac{\sqrt 3}{ 4}$. 
So the total number of triangles is equal to $\frac{4}{3} |z|^2 Q(a,b,c,d)$. Recall that $f(n)$ is the number of labelled triangulations $(\tilde T,A,B,C)$ with at most $2n$ triangles. The last condition is equivalent to $\frac{2}{3} |z|^2 Q(a,b,c,d) \le n$.

Let us study the conditions $\vec a + \vec b, \vec b + \vec c, \vec c + \vec d \in \Eis$.
In the case $z' \in \Eis \iff z \in \Eis_0$ the condition is satisfied automatically. 
Otherwise, $z' \notin \Eis \iff z \in \Eis_1 \sqcup \Eis_2$ the condition on their sums $\vec a + \vec b$, etc.,  belonging to $\Eis$ is equivalent to $a \equiv b \equiv c \equiv d \pmod 3$ by a direct computation.

Finally, we notice that the triangulations with $z$ and $ \omega z$ determine isometric triangulations.
This adds the factor $\frac{1}{6}$.  
\end{proof}

Given a lattice $L\subset \R^2$ and $\mathrm{Re}(s)>1$ we define the Epstein zeta function\index{Epstein zeta function} $$\zeta(L,s)=\sum_{\gamma\in L\setminus 0} <\gamma,\gamma>^s.$$ 

One can prove that $$\zeta(\Eis,s) =\sum_{z\in\Eis\setminus 0}|z|^{-2s}= 6\zeta_{\mathbb Q[\sqrt{-3}]}(s)= 6\zeta(s)L(\chi_{-3},s).$$  We refer to \cite{henn2016hexagonal} for details.

Now we are ready to estimate $f(n)$.
\begin{theorem}
	 \[
		 f(n) = \frac{1}{20} \sqrt{3} \cdot\L( \frac{\pi}{3} ) \zeta^{-1}(4) \zeta(\Eis, 2) n^2 + O(n^{3/2})
	,\] 
	as $n \to \infty$ where
	
$$\frac{1}{20} \sqrt{3} \cdot\L( \frac{\pi}{3} ) \zeta^{-1}(4) \zeta(\Eis, 2) \approx $$
$$\approx 0.20874321250560157071750716031497138622997487996283...$$  

Here $\zeta(s)$ is Riemann's zeta function.

\end{theorem}

\begin{proof}
By the definition, 
\begin{gather*}
	6 f(n) =
	\sum \limits_{z \in E_0} g(\frac{3}{2} |z|^{-2} n) + 
	\sum \limits_{z \in E_{\ne 0}} h(\frac{3}{2} |z|^{-2} n) = \\ 
	\sqrt{3} \L( \frac{\pi}{3} ) \frac{9}{4}\big [\sum \limits_{z \in E_0} (|z|^{-4} n^2+O(|z|^{-2} n)^{3/2}) + 
	\sum \limits_{z \in E_{\ne 0}}( \frac{1}{27} |z|^{-4} n^2 +O(|z|^{-2} n)^{3/2})\big ]= \\ 
	\sqrt{3} \L( \frac{\pi}{3} ) \frac{9}{4} n^2 (
	\sum \limits_{z \in E_{0}} |z|^{-4} + 
	\frac{1}{27} \sum \limits_{z \in E_{\ne 0}} |z|^{-4})+ \text{ ``error term'' }
\end{gather*}

The error term can be estimated as follows: 
$$\sum\limits_{z\in E_0\cup E_{\ne 0}} O((|z|^{-2} n)^{3/2})\leq cn^{3/2} \sum_{z\in \Eis}|z|^{-3}= O(n^{3/2}).$$

To compute the summands notice that 
\begin{gather*}
	\sum \limits_{z \in E_0} |z|^{-4} + 
	\sum \limits_{z \in E_{\ne 0}} |z|^{-4} = \\
	\sum \limits_{z \in \Prim \Eis } |z|^{-4} = \zeta^{-1}(4) \sum \limits_{z \in \Eis \setminus 0 } |z|^{-4} = 
	\zeta^{-1}(4) \zeta(\Eis, 2) 
.\end{gather*}

Indeed, $$\sum \limits_{z \in \Eis \setminus 0 } |z|^{-4} = \sum\limits_{k\in \mathbb Z_{>0}}\big[\sum \limits_{z' \in \Prim \Eis } |kz'|^{-4}\big] =  \zeta(4) \sum \limits_{z' \in \Prim \Eis } |z'|^{-4}$$ since for each vector $z\in \Eis$ there exists $k\in \mathbb Z_{>0}$ and $z'\in \Prim\Eis$ such that $z=kz'$.

Now we use \eqref{very_important_fact} which implies
\begin{gather*}
	\sum \limits_{z \in E_0} |z|^{-4} + \frac{1}{81}
	\sum \limits_{z \in E_{\ne 0}} |z|^{-4} = 
	\sum \limits_{z \in E_0} |z|^{-4} + 
	\sum \limits_{z \in E_{\ne 0}} |3z|^{-4} = \\
	=\sum \limits_{z \in E_0} |z|^{-4} + 
	\sum \limits_{z \in 3E_{\ne 0}} |z|^{-4} = 
	\sum \limits_{z \in \Prim \Eis_0} |z|^{-4} = \\
	=\zeta^{-1}(4) \sum \limits_{z \in \Eis_0 \setminus 0} |z|^{-4} = 
	\zeta^{-1}(4) \sum \limits_{z \in \Eis \setminus 0} |(1 + \omega) z|^{-4} =\\
	=\frac{1}{9} \zeta^{-1}(4) \sum \limits_{z \in \Eis \setminus 0} |z|^{-4} = 
	\frac{1}{9} \zeta^{-1}(4) \zeta(\Eis, 2) 
.\end{gather*}

From this system of linear equations one finds 
\begin{gather*}
	\sum \limits_{z \in E_0} |z|^{-4} = \frac{1}{10}
	 \zeta^{-1}(4) \zeta(\Eis, 2) \\
	\sum \limits_{z \in E_{\ne 0}} |z|^{-4} = \frac{9}{10}
	\zeta^{-1}(4) \zeta(\Eis, 2) 
\end{gather*}

It follows that

	$$f(n) = \frac{1}{6} \sqrt{3} \cdot\L( \frac{\pi}{3} ) \frac{9}{4}n^2 ( \frac{1}{10} + \frac{1}{27} \frac{9}{10}) \zeta^{-1}(4) \zeta(\Eis, 2) + O(n^{3/2})=$$ 
	$$=\frac{1}{20} \sqrt{3}\cdot \L( \frac{\pi}{3} ) \zeta^{-1}(4) \zeta(\Eis, 2) n^2 + O(n^{3/2}).$$

\end{proof}

\section{Examples and computer computations}
\label{sec_comp}

It follows from an Euler characteristic computation that no triangulation of $\R P^2$ with an odd number of triangles exists. 

Here is the list of $f(2n)-f(2n-1)$, i.e., the number of labelled triangulations $(T,A,B,C,\Delta)$ of $\R P^2$ with exactly $2n$ triangles, for $n=1,\dots, 74$:

$$0, 1, 4, 0, 16, 1, 12, 17, 20, 0, 46, 8, 18, 34, 40,$$
$$ 12, 64, 9, 36, 48, 60, 6, 94, 41, 24, 64, 72, 24, 112, 8,$$
$$ 60, 81, 94, 24, 160, 56, 42, 82, 114, 24, 160, 58, 60, 126, 96,$$
$$ 30, 190, 60, 96, 81, 160, 54, 184, 65, 72, 194, 132, 24, 238, 96,$$
$$ 90, 130, 220, 60, 232, 62, 84, 192, 214, 24, 286, 105, 90, 160.$$

Only one triangulation of $\R P^2$ with four triangles exists, see Figure~\ref{fig004}.

\begin{figure}[h]
\includegraphics[scale=0.1]{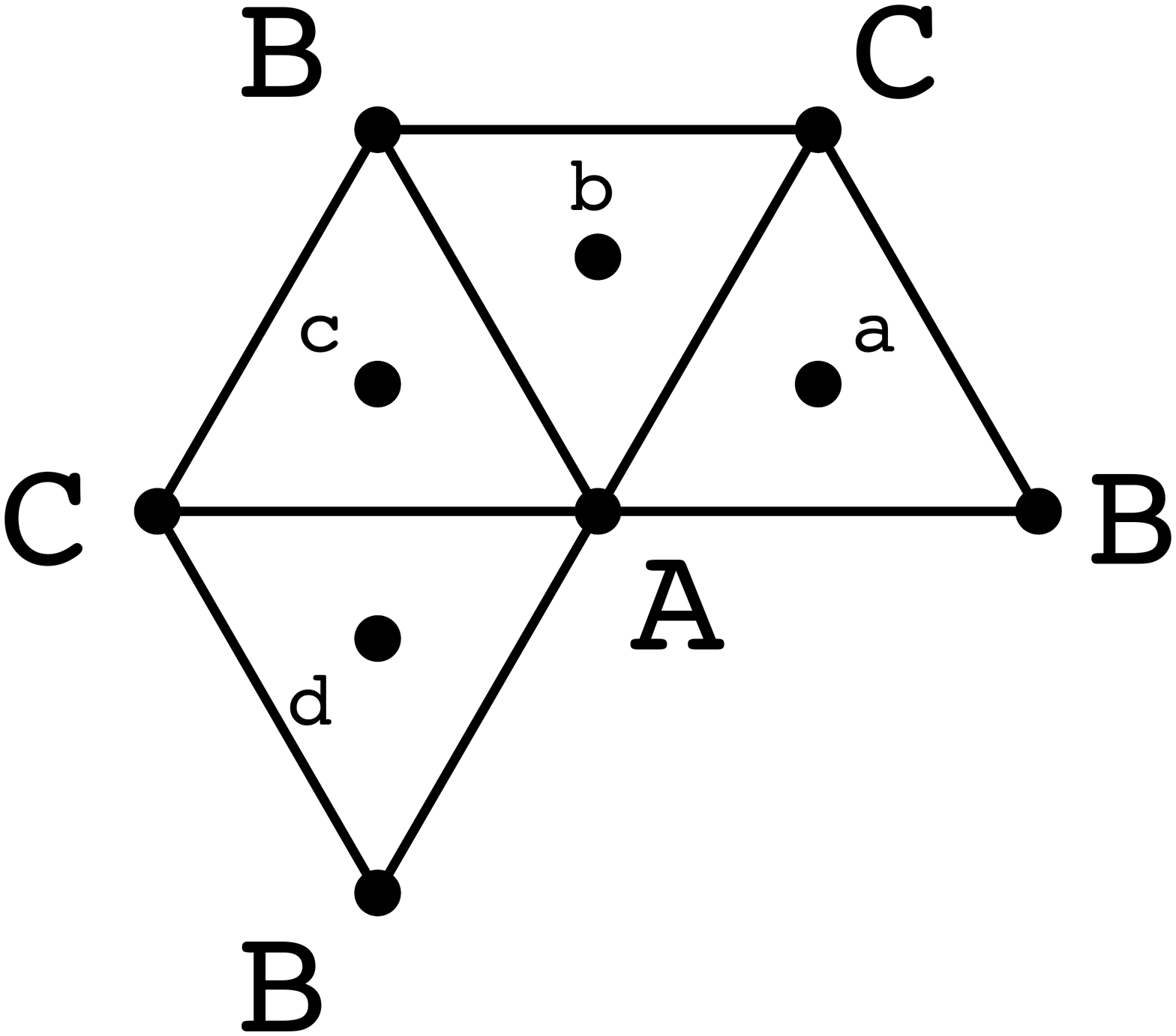}
\caption{In this case $z=\frac{1+\omega}{3}=\frac{1}{1-\omega^2}, (a,b,c,d)=(1,1,1,1)$.}
\label{fig004}
\end{figure}

Four marked triangulations of $\R P^2$ with $6$ triangles exist, see Figure~\ref{fig005}. 
\begin{figure}[h]
\includegraphics[scale=0.1]{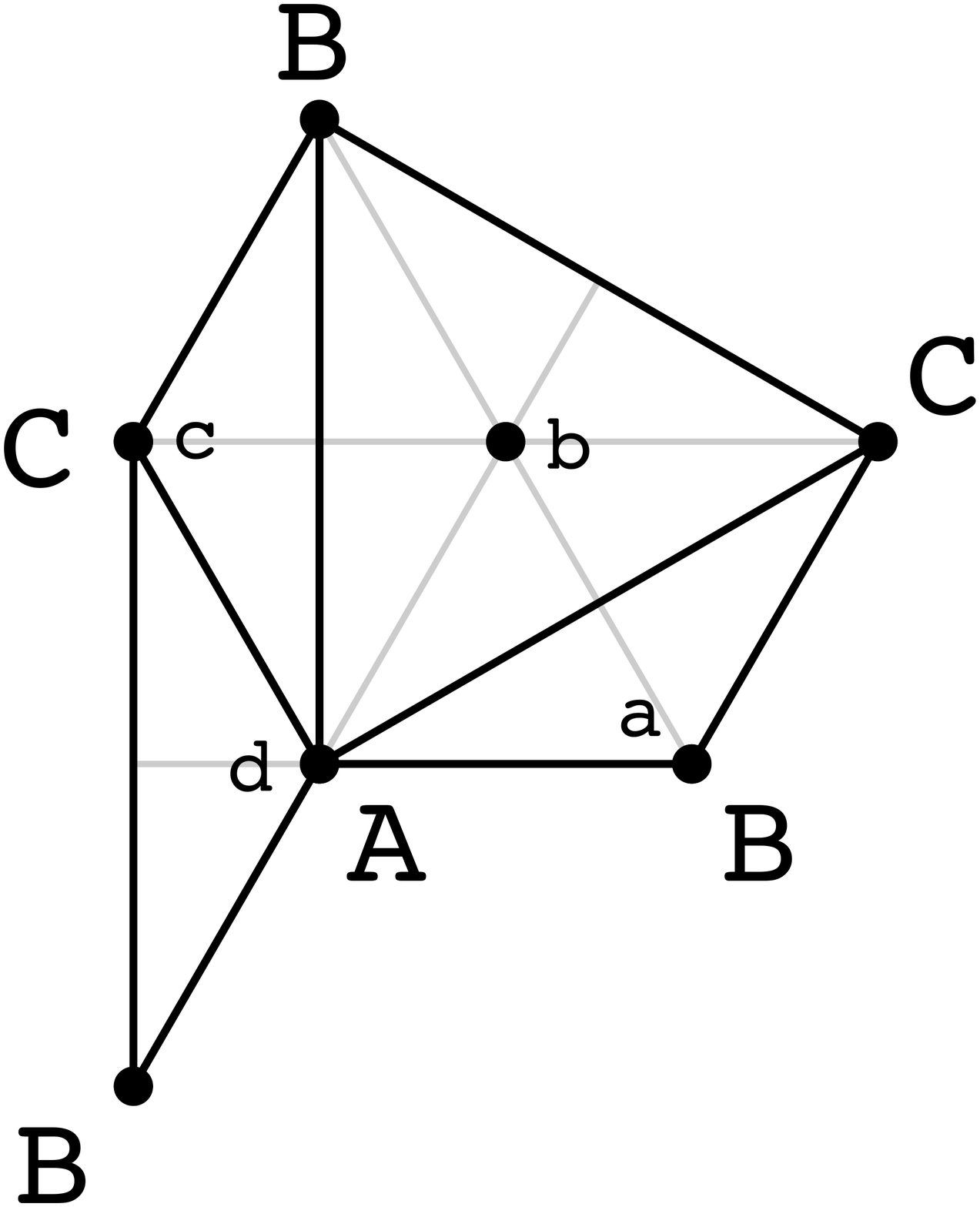}
\caption{In this case $z=1,(a,b,c,d)=(1,1,1,0)$, and we count this triangulation four times as $(a,b,c,d)=(1,1,1,0),(1,1,0,1),(1,0,1,1),(0,1,1,1)$. Two of them are isometric while another two differ by relabelling $B\to C$.}
\label{fig005}
\end{figure}

No triangulation of $\R P^2$ with $8$ triangles exists.

Sixteen triangulations with $10$ triangles exist, see Figure~\ref{fig006}.
\begin{figure}[h]
\includegraphics[scale=0.1]{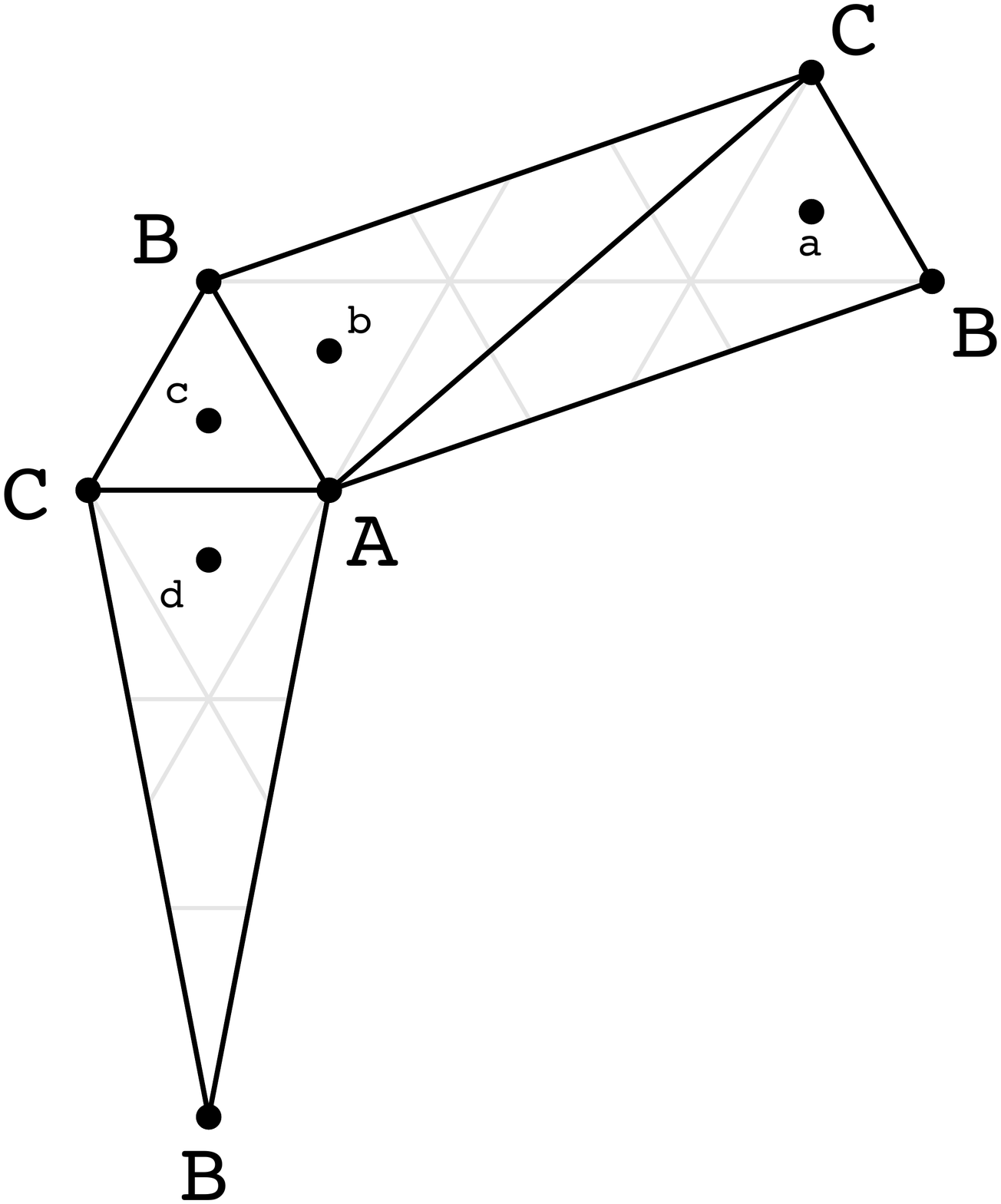}
\includegraphics[scale=0.1]{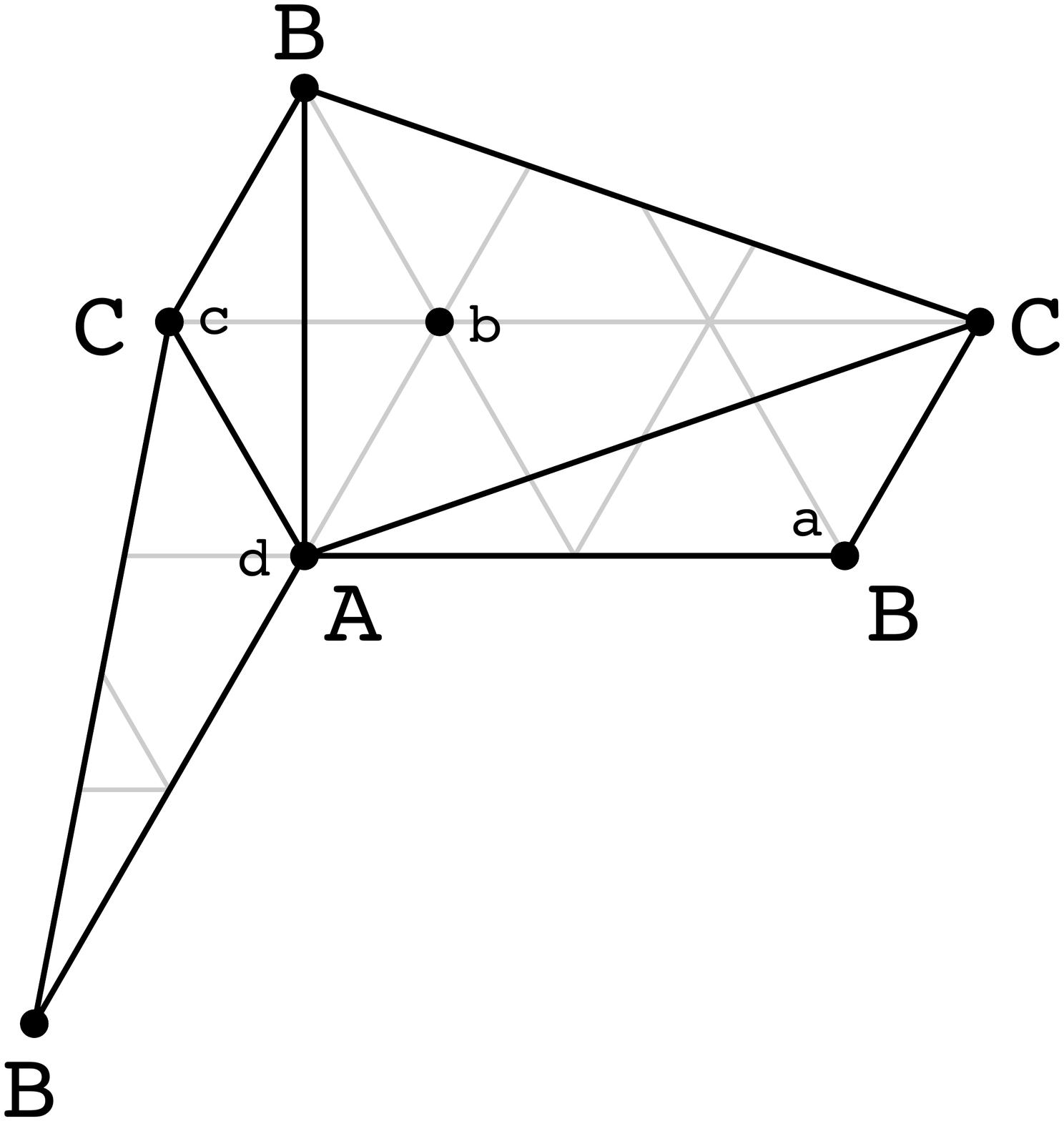}
\includegraphics[scale=0.1]{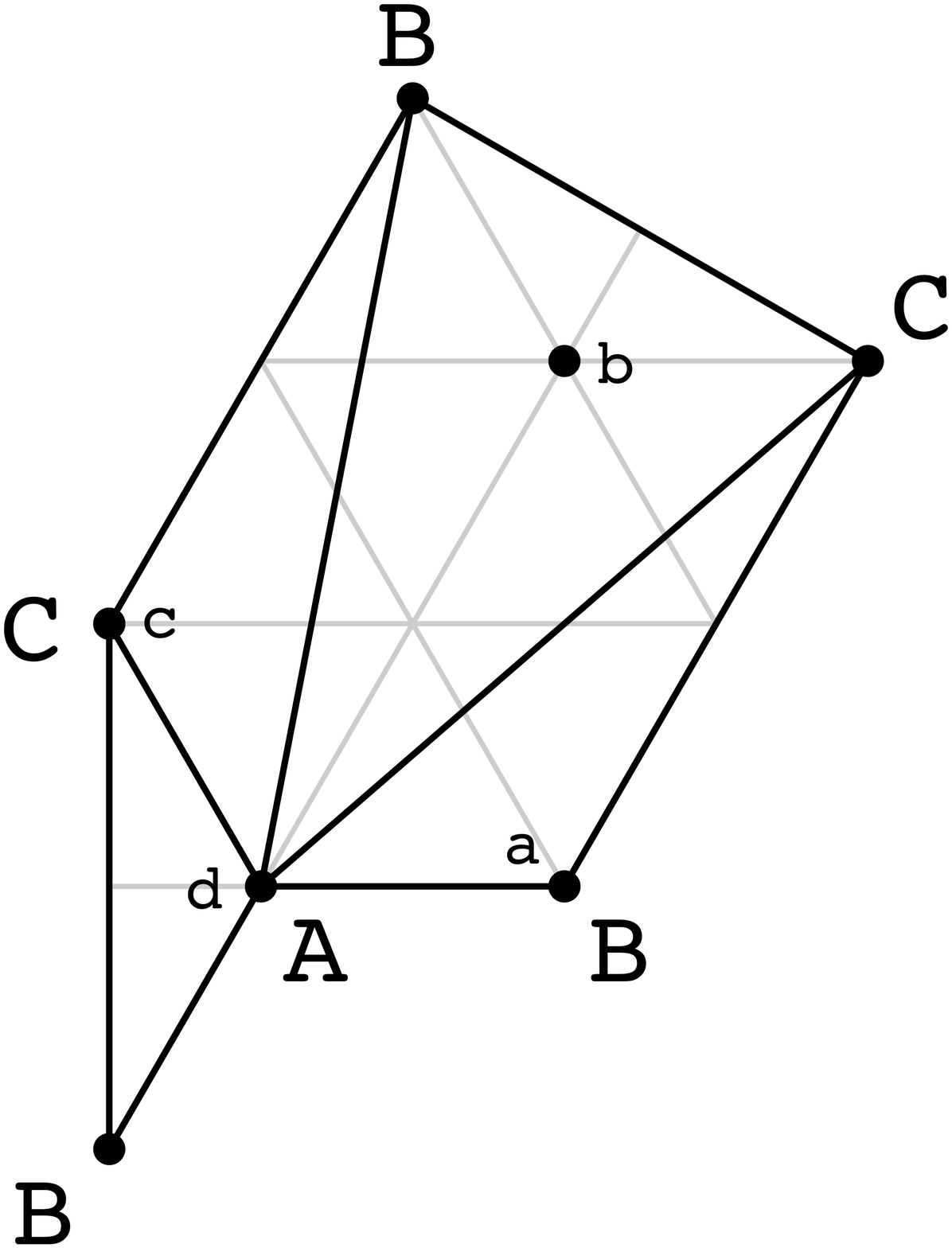}
\caption{The leftmost picture: $z=\frac{1+\omega}{3},(a,b,c,d)=(4,1,1,1)$ (counted four times). The central pictures and the rightmost picture are representatives for the tuple $(2,1,1,0)$ (counted $12=8+4$ times). Namely, $z=1,(a,b,c,d)=(2,1,1,0)$ in the central picture (counted eight times). The rightmost picture: $z=1,(a,b,c,d)=(1,2,1,0)$ (counted four times).}
\label{fig006}
\end{figure}
Only one triangulation with $12$ triangles exists, see Figure~\ref{fig002}.
\begin{figure}[h]
\includegraphics[scale=0.1]{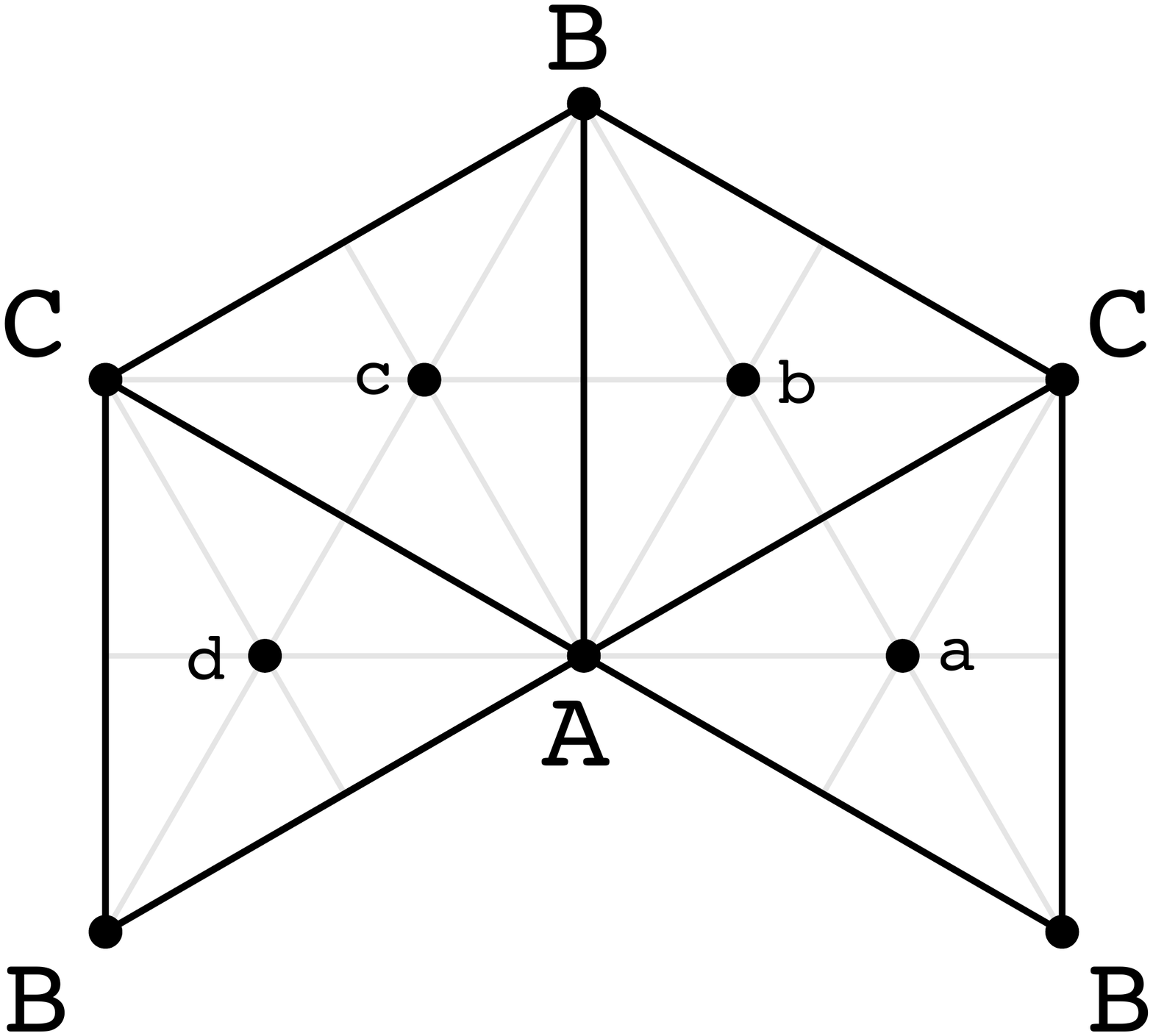}
\caption{In this case $z=1,(a,b,c,d)=(1,1,1,1)$.}
\label{fig002}
\end{figure}

\begin{figure}[h]
\includegraphics[scale=0.25]{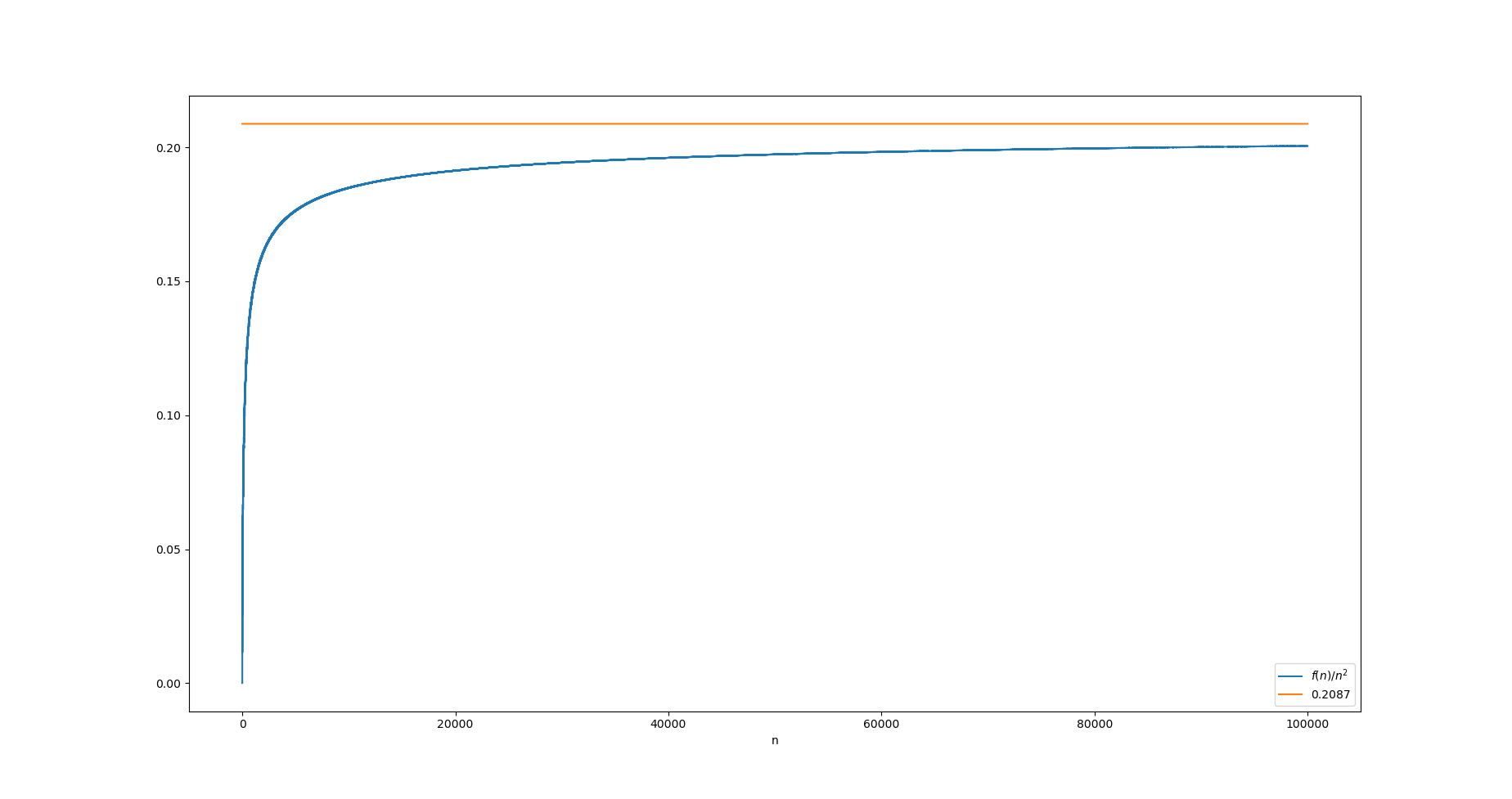}
\caption{On this plot we see that $f(n)/n^2$ converges to $C=0.2087...= \frac{1}{20} \sqrt{3} \cdot\L( \frac{\pi}{3} ) \zeta^{-1}(4) \zeta(\Eis, 2)$}
\end{figure}

\begin{figure}[h]
\includegraphics[scale=0.3]{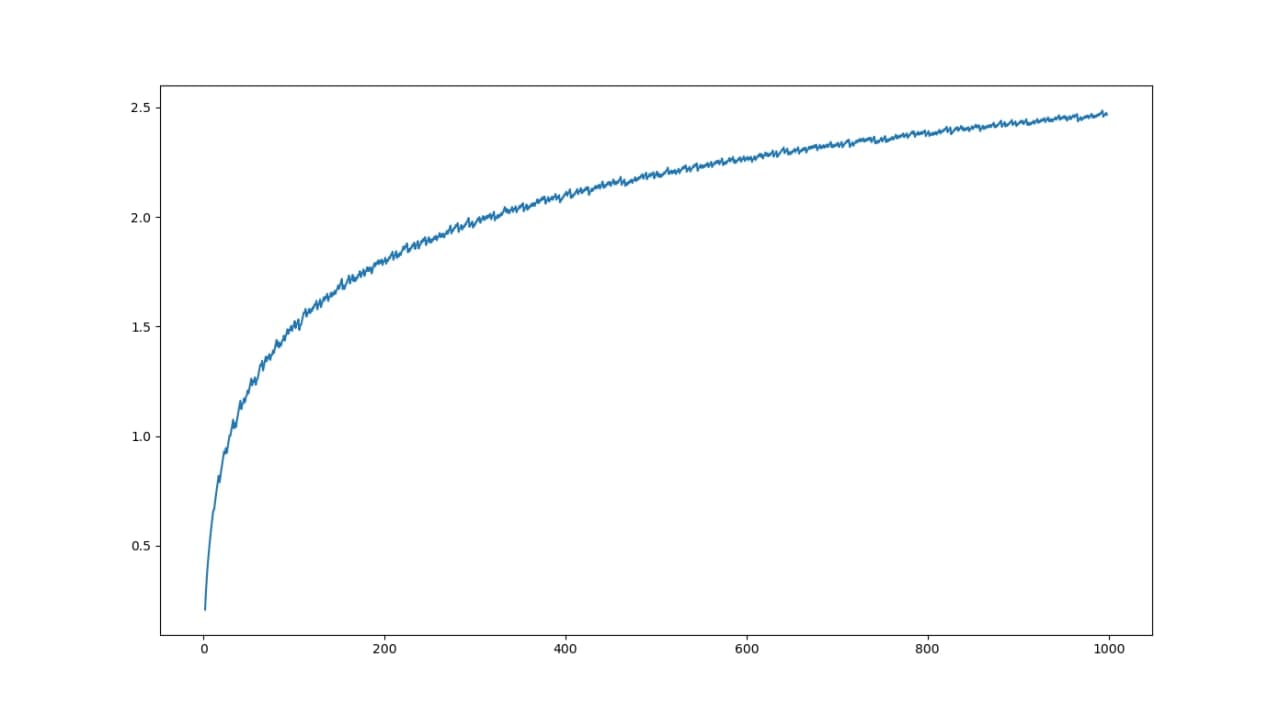}
\caption{The plot for the error term $\frac{1}{20} \sqrt{3}\L( \frac{\pi}{3} ) \zeta^{-1}(4) \zeta(\Eis, 2)n^2-f(n)$ divided by $n^{3/2}$ is presented. Thus we see that $f(n) \approx \frac{1}{20} \sqrt{3} \cdot\L( \frac{\pi}{3} ) \zeta^{-1}(4) \zeta(\Eis, 2)n^2-2.5n^{3/2}$.}
\end{figure}

\newpage
\section{Acknowledgments}

We all are grateful to the Euler Institute for its hospitality and to The Great Mathematical Workshop for the opportunity to start this project. Mikhail Chernavskikh is supported by the Basis Foundation scholarship. Research of Alexander Zakharov is supported by Ministry of Science and Higher Education of the Russian Federation, agreement № 075–15–2019–1619.

The research of Nikita Kalinin is supported by the Russian Science Foundation grant \textnumero 20-71-00007. Theorem 3.1 and Theorem 4.2 have been obtained under support of the RSF grant \textnumero 20-71-00007. We thank the referee for the careful reading of this chapter.

\bibliographystyle{plain}


\end{document}